\documentclass[11pt]{amsart}
\usepackage{amscd}
\usepackage{amsfonts}
\usepackage{color}

\usepackage{amsmath,amssymb,amsthm,latexsym}

\usepackage{amscd}

\usepackage{tikz-cd}

\usepackage{multicol}

\usepackage{enumerate}

\newtheorem{theorem}{Theorem}[section]

\newtheorem{proposition}[theorem]{Proposition}

\newtheorem{definition}[theorem]{Definition}

\newtheorem{corollary}[theorem]{Corollary}

\newtheorem{lemma}[theorem]{Lemma}

\usepackage{amsmath}

\usepackage{amsfonts}

\usepackage{amsmath,amsthm}

\usepackage{amscd}

\usepackage[all]{xy}

\numberwithin{equation}{section}

\theoremstyle{remark}

\newtheorem{remark}[theorem]{Remark}

\newcommand{\R}{\mathbb{R}}

\newcommand{\C}{\mathbb{C}}

\newcommand{\D}{\mathbb{D}}

\newcommand{\E}{\mathbb{E}}

\begin{document}

\title[Different base points for DPW]{\bf{Dependence of the loop group method on the base point}}
\author{Josef F. Dorfmeister}
\thanks{
}

\maketitle

\begin{abstract}
	We discuss, how the loop group method for harmonic maps $f: M \rightarrow \mathcal{S}$ 
	into symmetric spaces depends on the choice of  a base point $z_* \in M$.
	
	Actually, we consider two ways of discussing this dependence: 
	
	$\bullet$ The first way follows strictly the procedure of \cite{DPW}, but we do not assume from the beginning that the symmetric space $\mathcal{S}$ is realized a proiri as a fixed quotient $G/K$
	
	In this approach one chooses first a basepoint, say $z_0$ and then considers the natural realization 
	of the symmetric space $\mathcal{S}$  as $\mathcal{S} = G/K$, where $K$ denotes the fixed point 
	group of $G$ at the point $f(z_0) \in \mathcal{S}.$ 
	
		Given another base point $z_1 \in M$, one makes the analogous choices.
	
	The result of the first part of the paper is a dictionary, how to translate the description/realization
	  of $f$  relative to $z_0$ to the one relative to  $z_1$.

	$\bullet$ The second way, more convenient for explicit computations, see, e.g., \cite{Wang-a}, 
	is discussed in the second part of the paper. 
	
	In this case one realizes the symmetric space $\mathcal{S}$ once and for all as a fixed quotient space.
	We show that the corresponding frames are related by dressing. 
	
	In both cases we explain what these realizations mean for the construction of the compact dual of the symmetric space $\mathcal{S}$.
	
	To avoid repeated distinctions of cases we assume throughout that $\mathcal{S}$ is an inner 
	symmetric space. Moreover, since all frames and potentials are defined only on the universal 
	cover $\D$ of $M$, we use $M = \D$ unless the opposite is stated explicitly.
	\end{abstract}

%

      \vspace{2mm}

{\bf\   ~~ MSC(2010): \hspace{2mm} 53A30, 53C30, 53C35}

\section{Introduction}

The ``DPW-procedure"  \cite{DPW} has been applied to the discussion of the construction of many types of surfaces ( with or without symmetries), considered as submanifolds of certain specific manifolds. More precisely, we consider surface
classes  for which a harmonic  "Gauss type map" exists which characterizes the given surface of its class.
Examples for this are, among many others,  the surfaces of constant mean curvature in $\R^3$, 
(CMC surfaces in $\R^3$) \cite{DoHaMero},  the pseudo-spherical surfaces in $\R^3,$ \cite{DIS}
the minimal surfaces in the Heisenberg group \cite{DoInKo:Nil-Asian}, the indefinite proper affine spheres \cite{DoEi}, the ``unfashionable geometries"\cite{Bu-Jer}.

For the discussion of surface classes one considers primarily  harmonic ``Gauss type "  maps into semi-simple symmetric spaces. In this note we consider only cases, where the associated symmetric space is inner.

It is natural to normalize the setting by  choosing, once and for all, a fixed base point $z_0$ 
in the simply connected Riemann surface $\D$ of definition of the harmonic map $f: \D \rightarrow G/K$  and to require (w.l.g. up to congruence) that $f(z_0)= eK$ holds, where $e$ is the identity element of $G$.

The procedure suggested in \cite{DPW} was quite successful in many ways.

But it was never discussed, in what way the discussion needs to be adapted, if the fixed base point is chosen different from the originally chosen one.

This note gives two answers to this issue:

$\bullet$ First, by just following the DPW-procedure for different base points and describing the natural "isomorphisms",
from any type of geometric quantity related to one base point to the corresponding one of the second base point.

$\bullet$ Second, since a base point singles out a realization of the given symmetric space $\mathcal{S}$
 as a quotient space,  one can therefore ask, how one can handle a new base point in the environment of a 
 fixed   realization  $\mathcal{S} = G/K$.

In the first section we present a slightly modified loop group procedure for harmonic maps into inner
semisimple symmetric spaces, by starting from a base point $z_0$ and then realizing 
$\mathcal{S} \cong G/K$ with $K$ the stabilizer subgroup of $G$ relative to $f(z_0$.

 In section 2 we consider a second base point and compare the two loop group procedures associated to 
 these two  basepoints.
 
In section 3 we discuss how one can describe the loop group procedures for two different base points, if the 
symmetric space is realized in the form $G/K$ for both basepoints.

In the end we show how the harmonic maps into the compact dual symmetric space \cite{DoWa:Dual}
change with the choice of the base point.




\section{The (slightly generalized) loop group method for harmonic maps  $f:\D \rightarrow \mathcal{S}$}

In this section we assume (up to the last subsection)  that $\D$ is a contractible Riemann surface,
 i.e. the unit disk $\mathbb{E}$ or the whole  complex plane $\C$.
The last subsection will be devoted to the case $\D = S^2$.

If one considers harmonic maps defined on some Riemann surface $M$ which have nontrivial topology, 
then in the loop group approach one needs to consider the associate family of the harmonic maps on the universal cover $\D$ of $M$. 
Note that in general, a harmonic map of the associate family will not descend to the original Riemann surface $M$.
(There are some exceptions to this principle, but we can ignore them for the purposes of this paper.)
\subsection{The  (slightly generalized) basic DPW procedure for harmonic maps defined on contractible Riemann surfaces} \label{sect:1}


 
Let $\D$ be a contractible Riemann surface (different from $S^2$ ) and $\mathcal{S}$ a 
(connected) semisimple inner symmetric space, compact or non-compact.
The definition involves a fixed (connected) real semisimple Lie group G and a stabilizer subgroup $K$ such that
$\mathcal{S} \cong G/K$. For the purposes of this note we can assume that $K$ is connected.

 For more information about such spaces we refer to \cite{Berger}.
 
Finally, let $f: \D \rightarrow \mathcal{S}$ be a harmonic map.

We choose once and for all a base point $z_0 \in \D$ and obtain a corresponding 
base point $f(z_0) \in \mathcal{S}$.

To make connection with the ``classical" loop group method \cite{DPW} 
we consider the stabilizer $K_0 = K(z_0)$ of $f(z_0)$ in $G$, where
$G$ is a connected real semisimple Lie group acting transitively on the symmetric space $\mathcal{S}$
containing the group of isometries relative to the pseudo-Riemannian metric associated to $\mathcal{S}$.

Then the diagram 

\[
\begin{tikzcd}[column sep=4em,row sep=4em]
G  \ar{r}{\beta_0}  
    \ar{d}[swap]{ \pi_0}   &   \mathcal{S}     \\
G/{K_0}   \ar{ru}{j_0} 
\end{tikzcd}
\]

is commutative, where $\beta_0(g) = g.f(z_0)$ for $g \in G$, $\pi_0$ is the natural projection and $j_0$ is the 
natural choice of making the diagram commutative.
Note that $j_0$ is a diffeomorphism and equivariant relative to the actions of $G$ on
$G/{K_0}$ and $\mathcal{S}$ respectively.

We put $f_0 = j_0^{-1} \circ f$ and  obtain the harmonic map $f_0 : \D \rightarrow G/{K_0}.$ 
Note that $f_0$ satisfies

\begin{equation}
f_0(z_0) = eK_0 \in G/{K_0},
\end{equation}

where $e$ denotes the identity element of $G.$

At this point we have obtained the setting usually used for the discussion of harmonic maps by the loop group method.

We recall now briefly the steps of the loop group method, since we will need the notation and further 
details in the next subsection.

\subsection{Step 1: } Choose any frame $F_0(z, \bar z) : \D \rightarrow \Lambda G_\sigma$ of $f_0(z, \bar z)$ taking values in 
$G$ for all $z \in \D$ and satisfying  $F_0(z_0,{ \bar z}_0) = e.$

Note that this frame then satisfies for all $z \in \D$
$$ F_0(z, \bar z).eK_0  = f_0(z, \bar z).$$

 Next we consider the Maurer Cartan form of $F_0$,

\begin{equation} \label{MC of F_0}
\alpha_0 = F_0^{-1} dF.
\end{equation}

We will decompose this differential one-form below relative to the symmetry of the 
symmetric space $\mathcal{S}$.
More precisely, we consider the involutive automorphism $\sigma_0$ of $G$, 
equivalently of $Lie G,$ 
which defines the symmetric space structure of $G/{K_0}.$ In particular we have 
$Lie K_0 \subset  Fix_\sigma  (Lie G).$

\begin{remark}
According to what was said above we assume for our purposes in this paper that $K_0$ is
the connected component ${Fix_{\sigma_0}  (Lie G)}^0$ of $ Fix_{\sigma_ 0} (Lie G)$, since all 
other harmonic maps can be obtained by projection to the corresponding symmetric spaces.
\end{remark}

Putting 
$$\mathfrak{g} = Lie G \hspace{2mm} \mbox{and} \hspace{2mm} \mathfrak{k}_0 = Lie K_0$$
we obtain  
$\mathfrak{k}_0 =  \{x \in  \mathfrak{g} ;  \sigma_0 (x) = x \}$.

Setting $\mathfrak{p}_0 =  \{x \in  \mathfrak{g}_0 ;  \sigma_0 (x) = -x \}$  we derive

\begin{equation}
 \mathfrak{g} =  \mathfrak{k}_0  \oplus \mathfrak{p}_0. 
\end{equation}

 Moreover, $\alpha_0$ decomposes as
 
 \begin{equation}
  \alpha_0 = F_{0}^{-1} dF = \alpha_0^{\prime} + \alpha_{0,k} +   \alpha_0^{\prime \prime}.
 \end{equation}

  where the summands $\alpha_0^\prime$ and $\alpha_0^{''}$ are
  contained in $ \mathfrak{p}_0 $ and $\alpha_{0,k} \in \mathfrak{k}_0$.

\subsection{Introducing the loop parameter and the associated family of harmonic maps associated to $f$ relative to $z_0$}

Next we introduce the loop parameter $\lambda.$ 

In  geometric statements, $\lambda$ 
will be assumed to be an element of the unit circle $S^1,$ but in many cases it will be important that certain functions and 1-forms which depend on $\lambda$  can actually be defined on  
some open subset of $\C$, like $\C^*.$ In this note we will not address this issue,
but refer to the standard literature about the loop group method.

For the introcuction of  the loop parameter $\lambda$ as in \cite{DPW} we extend the frame $F_0$ of a harmonic map $f_0$  to the extended frame.

For this we consider the "loopified" Maurer-Cartan form $\alpha_\lambda$ of $F$:

\begin{equation}
\alpha_{0,\lambda} = F_{0,\lambda}^{-1} dF =\lambda^{-1} \alpha_{0}^\prime + \alpha_{0,k} +  \lambda \alpha_{0}^{\prime \prime}.
\end{equation}

Since we know (see e.g. \cite{DPW} ) that $f_0$ is harmonic if and only if $\alpha_{0,\lambda} $ 
is integrable for all  $\lambda \in S^1$ (equivalently all $\lambda \in \C^*$), we infer

\begin{lemma}
With $f, f_0, F_0$ and $\alpha_{0,\lambda}$ as above we can solve  the differential equation below for all $\lambda \in \C^*$
$$d F_{0,\lambda} = F_{0,\lambda} \alpha_{0,\lambda} \hspace{2mm} \mbox{satisfying} \hspace{2mm}
F_{0,\lambda}(z, \bar z) = e.$$
\end{lemma}

The frame $F_{0,\lambda}$ is called ``the extended frame associated to $f_0$ relative to $z_0$".

We frequently write $F_{0,\lambda}(z, \bar z) = F_0(z, \bar z, \lambda)$ and analogously for other 
maps and 1-forms.

From here on we will usually consider  $\lambda-$dependent quantities and therefore use, by abuse of notation, $F_0$ and $F_{0,\lambda}$  synonymously. We trust that this will cause no problem to the reader.

The following observation is one of the cornerstones of the loop group method.
For notation we refer to the appendix.

\begin{lemma}
{The extended frame}  $F_0(z, \bar z, \lambda) $ 
 {is contained in the real twisted loop subgroup}  $\Lambda  {G}_{\sigma_0}$
{of}  $ \Lambda  {G}_{\sigma_0}^\C.$
\end{lemma}

Finally, we introduce  the associated family of harmonic maps $f_{0,\lambda}$  associated to $f$ relative to $z_0$.

\begin{definition} \label{associative family}
Let $f, f_0$ and $F_{0,\lambda}$ be as above.Then we put for $\lambda \in \C^*$
$$ f_{0,\lambda} = F_{0,\lambda} mod K_0.$$
\end{definition}

We can state now (see \cite{DPW}):

\begin{proposition}
Attaining the assumptions and definitions made above we have:

$(1)$  The family of maps $ f_{0,\lambda}$ consists, for $\lambda \in S^1,$  of harmonic maps 
from $\D$ into the inner symmetric space $G/{K_0}$.

$(2)$ For any fixed $\mu \in S^1$ the extended frame $ F_{0,\lambda} , \lambda \in S^1,$
is the extended frame of the harmonic map  $ f_{0,\mu}.$
In particular,  for all $\mu \in S^1$ the frame $ F_{0,\mu}$ is a frame of the harmonic map $ f_{0,\mu}$  satisfying  
$ F_{0,\mu}(z_0, \overline{z_0}) = e.$

The family of harmonic maps$ f_{0,\lambda}, \lambda \in S^1$  is called  {\bf the associated family of harmonic maps $f_{0,\lambda}$  associated to $f$ relative to $z_0$}.
\end{proposition}

\subsection{The loop group method}

\subsubsection{From harmonic maps to normalized potentials}

We will consider $f, f_0, F_0$ etc. as above.

So far we have constructed $F_{0,\lambda}$ and $f_{0,\lambda}$.

The heart of the loop group method is to implement the following flow chart (and its converse):

$$ f_{0,\lambda} \rightsquigarrow  F_{0,\lambda}   \rightsquigarrow F_{0,-} \rightsquigarrow  \eta_{0,-}  $$

The first arrow follows from the proposition just above.

The second arrow refers to the (meromorphic) ``normalized extended frame"  $F_-(z, \lambda).$ 

It is obtained from the extended frame $F_0(z, \bar z, \lambda)$ by the (unique) Birkhoff splitting 
away from a discrete subset of $\D$
\begin{equation}
F_0(z, \bar z, \lambda) = F_{0,-}(z,  \lambda) L_{0,+}(z, \bar z, \lambda),
\end{equation}
where $F_{0,-}$ is chosen w.l.g. to have the form $F_{0,-}(z, \lambda) = e + \mathcal{O}(\lambda^{-1})$.

The uniqueness of the splitting implies $F_{0,-}(z_0, \lambda) = e$.

Note, by the general theory,  $F_{0,-}$ is meromorphic in $z$, whence $\bar z$ does not occur in the normalized
extended frame.

Finally, we define the normalized potential \cite{DPW} associated to $f$ relative to $z_0$.

\begin{definition}
The normalized potential  $\eta_{0,-}$ of the harmonic map $f$  relative to the base point $z_0$ is the 
Maurer-Cartan form of the normalized extended frame  $F_{0,-}(z, \lambda)$:

$$  \eta_{0,-}(z, \lambda) = F_{0,-}(z, \lambda)^{-1} d F_{0,-}(z, \lambda).$$
\end{definition}

It is  easy to see that $  \eta_{0,-}(z, \lambda) $ is of the form

$$   \eta_{0,-}(z, \lambda) = \lambda^{-1}  \xi_{0,-}(z, \lambda) dz,$$
where $\xi_{0,-}(z, \lambda)$ is meromorphic on $\D$ and contained in $\mathfrak{p}_0^\C.$ 

Also note, since our construction yields  $F_{0,-}(z_0, \lambda) = e$, the coefficient 
$\xi_{0,-}(z, \lambda)$ of the normalized potential  does not have a pole at $z_0$.

\subsubsection{From normalized potentials to harmonic maps }

This section mainly consists of reversing the steps of the last subsection.

Let $$  \hat{\eta}_{0,-}(z, \lambda) = \lambda^{-1}  \xi_{0,-}(z, \lambda) dz,$$
be a  meromorphic 1-form defined on $\D$ which takes values in   $\mathfrak{p}_0^\C$ and is finite
at $z_0.$  Then the differential equation

$$ d  \hat{F}_{0,-}(z, \lambda) =   \hat{F}_{0,-}(z, \lambda) \hat{\eta}_{0,-}(z, \lambda),   
\hat{F}_{0,-}(z_0, \lambda) = e $$

always has a holomorphic solution in a neighbourhood of $z_0$.

We will assume from here on that actually a meromorphic solution exists on all of $\D$.
We will call such one-forms ``normalized potentials". 

Next we will construct a frame  $\hat{F}_{0,\lambda} : \D \rightarrow \Lambda G_{\sigma_0}$.

\begin{proposition} [Iwasawa decomposition]
The meromorphic matrix function  $  \hat{F}_{0,-}(z, \lambda),$ defined  by assumption on $\D$, has
locally near $z_0$ a decomposition of the form

$$   \hat{F}_{0,-}(z, \lambda) =   \hat{F}_{0,}(z, \bar z, \lambda) \hat{V}_0(z, \bar z, \lambda)$$
with $\hat{F}_{0,\lambda} : \D \rightarrow \Lambda G_{\sigma_0}$ satisfying $\hat{F}_{0} (z_0, \bar{z}_0, \lambda) = e$ and 
$\hat{V}_{0,\lambda} : \D \rightarrow \Lambda^+ G^\C_{\sigma_0}.$

Here all matrix functions are defined and finite away from the poles of $ \hat{F}_{0,-}(z, \lambda)$.
\end{proposition}

The decomposition can be made unique, if one makes sure (w.l.g. this is possible) 
that the leading term of $\hat{V}_{0,\lambda}$ only has positive diagonal entries.

Finally we put

$$\hat{f}_0 (z, \bar z ,\lambda) = \hat{F}_0 (z, \bar z, \lambda) mod K_0.$$

\begin{theorem}
The family $\hat{f}_0 (z, \bar z ,\lambda)$ is a family of harmonic maps defined on $\D$ and 
takes values in $G/K$.

Furthermore, the two procedures outlined in this subsection and in the previous one are inverse to each other.
In particular, this gives a  1-1-relation between harmonic maps from $\D$ into $G/{K_0}$ and the space of
 normalized potentials for any basepoint $z_0 \in \D.$ 
\end{theorem}


\section{\large{Relating the construction relative to $z_0$ to the construction relative to $z_1$}}

In this section we retain the setting of the last section, but consider in addition the loop group method relative to another basepoint $z_1 \in \D, z_0 \neq z_1$.

Note, that the group $G$ has been chosen a priori together with the symmetric space $\mathcal{S}.$

Clearly, we can apply now, what was said in the previous subsections of this section about the loop group method relative to $z_0,$ verbatim to  the loop group method relative to $z_1$.

Notationally , we  will indicate by using the subscript $1$ in place of $0$ that the quantity under consideration 
belongs to the loop group procedure relative to $z_1$.
We trust that the reader will have no problems to adjust for this change of label.

\subsection{\large{A comparison of the geometric settings associated to $z_0$ and to $z_1$}}

As above, we consider  the stabilizer of $f(z_1)$ in $G$ and denote it by $K(z_1) = K_1.$
Then we have the analogous diagram

\[
\begin{tikzcd}[column sep=4em,row sep=4em]
G  \ar{r}{\beta_1}  
    \ar{d}[swap]{ \pi_1}   &   \mathcal{S}     \\
G/{K_1}   \ar{ru}{j_1} 
\end{tikzcd}
\]

 and obtain a natural identification 
$G/{K_1} \rightarrow \mathcal{S}$ via $j_1$.

To relate the two loop group procedures we choose  some $h \in G$ satisfying

\begin{equation} \label{zo->z1}
h.f(z_0) = f(z_1).
\end{equation}

Note that here "." denotes the action of $G$ on $G/K$ given by: 
$g.(wK) = (gw)K.$

\begin{remark}
For later purposes we would like to point out that $h$ is uniquely determined up
 to multiplication on the right by an element of $K_0$ and up to multiplication on the left 
 by an element of $K_1.$
 \end{remark}

We observe that the isotropy group $K_1$ of $f(z_1)$ satisfies

\begin{equation} \label{isostabi}
K_1 = h K_0 h^{-1}.
\end{equation}
 Indeed, $b.f(z_1) = f(z_1) \Longleftrightarrow bh.f(z_0) = h.f(z_0) \Longleftrightarrow
 h^{-1} b h \in K_0.$

 For the involution ${\sigma_1}$ of $G$ defining the symmetric space $\mathcal{S} \cong G/K_1$ we obtain
 
 \begin{equation} \label{trafosigma}
 {\sigma_1} = Ad(h){ \sigma_0} Ad(h)^{-1},
 \end{equation}
 where $Ad(g)$ denotes the adjoint representation of $g \in G$.
 
 Furthermore we have (in the sense of \cite{Berger}, section 7):
 
 \begin{proposition} \label{isoquotients}
In view of   (\ref{isostabi}), the isomorphism of Lie groups, $Ad(h):
 G \rightarrow G, g \mapsto h g h^{-1},$ induces the isomorphism
of symmetric spaces
$$\varphi : G/K_0 \rightarrow   G/ K_1, gK_0 \mapsto  h g h^{-1} K_1= h (gK_0) h^{-1}.$$
\end{proposition}

 \subsection{A comparison of the loop group methods associated to $z_0$ and to $z_1$}

 We start by stating the ${\sigma_1}-$twisted loop group (which is associated to $f(z_1)$ and $\sigma_1$):

 \begin{equation}
 \Lambda G^{\C}_{{\sigma_1}} = \{ g(\lambda) \in  \Lambda G^{\C}; 
 {\sigma_1}(g)(\lambda) = {\sigma_1}(g(-\lambda)) = g(\lambda) \}.
 \end{equation}

 For the two twisted  loop groups associated to the two base points we infer the relation:

 \begin{lemma} \label{relation of loop groups}
  $\Lambda G^{\C}_{{\sigma_1}} = h \cdot  \Lambda G^{\C}_{\sigma_0} \cdot  h^{-1}.$
 \end{lemma}
 \begin{proof}
 $g \in    \Lambda G^{\C}_{{\sigma_1}} \Longleftrightarrow Ad(h) \sigma_0 Ad(h)^{-1} g(-\lambda) = g(\lambda)
  \Longleftrightarrow  \sigma_0 (Ad(h)^{-1} g(-\lambda) ) = Ad(h)^{-1} g(\lambda)  \Longleftrightarrow
   Ad(h)^{-1} g(\lambda) \in \Lambda G^{\C}_{\sigma_0}   \Longleftrightarrow   
   \Lambda G^{\C}_{{\sigma_1}} = h   \Lambda G^{\C}_{\sigma_0} h^{-1}.$
 \end{proof}
 
 Note that an analogous result holds for the groups   $\Lambda^{\pm} G^{\C}_{{\sigma_1}}$ and 
 $  \Lambda G_{{\sigma_1}}$.
 \vspace{2mm}

Next we consider  again the harmonic map $f_0:\D \rightarrow G/{K_0}$ discussed in section \ref{sect:1}
and define 

\begin{equation} \label{definehatf}
f_1: \D \rightarrow  G/K_1, \hspace{2mm}\mbox{by}\hspace{2mm} f_1 = {\varphi} \circ f_0.
\end{equation}

Note that $f_1$ is again harmonic, since ${\varphi}$ is an isometry relative to the non-degenerate forms 
induced by the Killing form of $G$.

The general theory tells us  that there exists a frame ${F_1}$ for the  harmonic map 
$f_1: \D \rightarrow G/K_1.$ More precisely we have
 
 \begin{lemma}
\mbox{There exists a frame} \hspace{2mm}$ {F_1}: \D \rightarrow G$
 \hspace{2mm} \mbox{for the harmonic map} \hspace{2mm} $ {f_1}: \D \rightarrow G/ K_1$
 
  \hspace{2mm} \mbox{such that} \hspace{2mm}
${f_1}(z, \bar z) = {F_1}(z, \bar z).K_1$  \hspace{2mm} \mbox{and } \hspace{2mm} 
${F_1}(z_1, \overline{z_1}) = e.$
\end{lemma}

Note, writing locally $f_0 = g_0K_0$ we obtain 
$$f_1 = \varphi \circ f_0 = h g_0K_0 h^{-1} = h g_0 h^{-1} K_1.$$

Hence, using the definitions of $F_0, {F_1} $ and  ${\varphi}$  we obtain in a straightforward way:

\begin{corollary}  \label{trafoframe}
${F_1}(z, \bar z) = h F_0(z, \bar z) k_0(z, \bar z) h^{-1} =( h F_0(z, \bar z)h^{-1}) \cdot ( h k_0(z, \bar z) h^{-1}),$

where $  h k_0(z, \bar z) h^{-1} \in K_1.$
\end{corollary}

\begin{lemma}
Using  ${\sigma_1}$ and the corresponding twisted loop groups we can now introduce 
 the loop parameter $\lambda$ as in Section \ref{sect:1}and  extend the frame ${F_1}(z, \bar z)$
 to the extended frame  ${F_1}(z, \bar z, \lambda)$ for all $z \in \D, \lambda \in \C^*$. It satisfies

\begin{equation}
F_1(z, \bar z, \lambda) \in  \Lambda  {G}_{{\sigma_1}} \hspace{2mm} \mbox{and } \hspace{2mm} 
{F_1}(z_1, \bar{z}_1, \lambda) = e.
\end{equation}
\end{lemma}

The next step of our discussion is to consider the (meromorphic) normalized extended frame 
$F_{1,-}(z, \lambda)$  for ${f_1}$.
It is obtained from the 
(unique) Birkhoff splitting
\begin{equation}
{F_1}(z, \bar z, \lambda) = F_{1,-}(z,  \lambda) L_{1,+}(z, \bar z, \lambda),
\end{equation}
where $ F_{1,-}$ is chosen w.l.g. to have the form $ F_{1,-}(z, \lambda) = e + \mathcal{O}(\lambda^{-1})$.
The uniqueness of the splitting implies $ F_{1,-}(z_1, \lambda) = e$.

Note that here the normalized potential $\eta_{1,-} =  F_{1,-}^{-1}d F_{1,-}$ is meromorphic and 
integrates to $ F_{1,-}$, 
if one chooses the intial condition $ F_{1,-}(z_1, \lambda) = e$.

In view of Lemma \ref{relation of loop groups}  and Corollary \ref{trafoframe} we obtain

\begin{theorem}
The following relations hold:

$\bullet$ $F_1(z, \bar z, \lambda) =( h F_0(z, \bar z,\lambda)h^{-1}) \cdot (h k_0(z, \bar z) h^{-1}),$

$\bullet$ $F_{1,-}(z,\lambda) = h F_{0,-}(z, \lambda) h^{-1},$

$\bullet$ $F_{1,+}(z, \bar z, \lambda) =( h F_{0,+}(z, \bar z,\lambda)h^{-1}) \cdot (h k_0(z, \bar z) h^{-1}).$
\end{theorem}
\begin{proof}
The first relation restates Corollary  \ref{trafoframe}. The remaining relations can be read off immediatly
from the Birkhoff decompositions of the first relation.

\end{proof}

\begin{remark}
The theorem just above presents a very natural and simple relation between the loop group methods
for harmonic maps relative to two base points $z_o$  and $z_1$. respectively.
\end{remark}


\subsection{\large{Relating maximal compact duals relative to different base points}}

Finally we discuss how the compact duals \cite{DoWa:Dual} of $G/K_0$ and $G/K_1$ are related, if $G/K_0$ is the inner
symmetric space related to the base point $z_0$ with isotropy group $K_0$ and $G/K_1$ is the 
isomorphic inner symmetric space related to the base point $z_1$ with isotropy group $K_1.$
Note that we assume now, as always in \cite{DoWa:Dual}, that $G$ is simply connected.
Moreover, we retain the notation  of the previous sections.

\subsubsection{\large{About the construction of a compact dual symmetric space relative to $z_0$}}

An inspection of the discussion of the paper \cite{DoWa:Dual} shows that  for the construction of the 
compact dual
we need three pairwise commuting involutive automorphisms of the complexification 
$G^\C$ of $G.$

$(1)$ The (anti-holomorphic) real form involution $\tau_0$ of the complex  Lie algebra $\mathfrak{g_0^\C}$
and its extension to $G^\C$.

We note that this involution is the same for all base points in $\D$.

$(2)$ The (holomorphic) inner symmetric space involutions $\sigma_0$ and ${\sigma_1}$ respectively.

We note that (\ref{trafosigma}) shows that these two involutions are related  by the automorphism
$Ad(h)$ of $G$ resp. $\mathfrak{g}.$ 

$(3)$ A (anti-holomorphic) Cartan involution $\theta$ of $\mathfrak{g}^\C.$

Recall again that $\tau_0$, $\sigma_0,$ and $\theta_0$ commute pairwise.
In section 3.1 of \cite{DoWa:Dual} we have explained how the compact dual is constructed  from such data.

We also know from the examples presented in section 6 of  \cite{DoWa:Dual} that, in general, 
the compact dual is not unique.


\subsubsection{\large{About the construction of a compact dual inner symmetric space relative to $z_1$}}

When considering the construction of a compact dual for the base point $z_1$ we use the symmetric space 
$G/K_1$.

$(1)^\star$ It is easy to see that the real form involution ${\tau_1}$ is equal to $\tau_0$. 
We will just write $\tau$ from here on.

$(*)$ Since $h$ is real, $\tau$ commutes with $Ad(h).$

$(2)^\star$ The symmetric space involution ${\sigma_1}$ for $z_1$ satisfies
 \begin{equation} 
 {\sigma_1} = Ad(h) \sigma_0 Ad(h)^{-1},
 \end{equation}
 by (\ref{trafosigma}).
 
 $(3)^\star$ As in the case of the basepoint $z_0,$ for the construction of a compact dual it suffices  to consider a Cartan involution $\theta_1$ which commutes with $ \tau$  and ${\sigma_1}$.

 \begin{theorem}
 For the construction of a compact dual for theinner  symmetric space relative to $z_1$ it suffices 
 to consider the three pairwise commuting involutions
 
 $(a)$ $\tau$,
 
 $(b)$ ${\sigma_1}  = Ad(h) \sigma_0 Ad(h)^{-1}$,
 
 $(c)$ ${\theta_1}  = Ad(h) \theta_0 Ad(h)^{-1}.$
 
 This correspondence yields a  $1-1-$relation between all compact duals relative to $z_0$ and 
 all compact duals relative to $z_1$.
  \end{theorem}

\begin{corollary}
By the correspondences  given in the  theorem just above all base points 
$z_* \in \D$ naturally yield isomorphic compact duals.
In particular, up to the correspondence above, all possible compact duals occur for any fixed base point.
\end{corollary}


\subsection{\large{The influence of a base point on the DPW procedure in the case  $\D= S^2$ }} \label{sectionS^2}
In this section we assume that $\D$ is the  Riemann surface $S^2.$

Let $f:S^2 \rightarrow G/K$ be harmonic.
Let's choose two points $z_0, z_1 \in S^2$ to be used as base points for the loop group method.

We assume w.l.g. that both points are contained in the equator of $S^2$ relative to the 
northpole and the southpole.

Let $V_S$ be the lower hemisphere extended by  a small  open annulus (containing the equator), 
basically the open disk $\D_{1 + \epsilon}^S$ of radius $1 + \epsilon$ and center $0 = $southpole.

Analogously we define $V_N$ or by reflection of $V_S$ across the equator.

Next we consider the pair of harmonic maps $f_S = f_{|V_S}$ and $f_N = f_{|V_N}$ separately 
and proceed for each of them with the loop group method relative to the points $z_0$ and $z_1$ 
respectively as in the preceding sections.

Consider, for the time being, the harmonic maps $f_S$. Then, as  above, we can choose an
$h^S \in G$ mapping $f_S(z_0)$ to $f_S(z_1)$.
For these choices the results of the previous sections apply.

Similarly, one can choose $h^N \in G$ mapping $f_N(z_0)$ to $f_N(z_1)$.
And also for these choices the results of the previous sections apply.

It is clear, that we can choose $h^N  = h^S,$  observing that this choice is independent of $\lambda$.

 Now it is easy to verify that the separate results for
 the pairs $(f_S, z_0 )$ and $(f_S, z_ 1)$ as well as the pairs  $(f_N, z_0 )$ and $(f_N, z_ 1)$ match up and yield the global relation for the pairs $(f_0, z_0)$ and $(f_1,z_1)$ on $S^2$.
 
 Therefore, for $\D = S^2$ we have the analogous  relation as before.
 
 We leave the details to the reader.

\section{An alternative way to handle different base points}

In addition to the classical DPW-approach, as explained in the last two sections,  
which handles different base points in a unified manner, there exist "modified" approaches
which can be phrased about as explained below. 
The main goal is, as above, to consider some harmonic map and to describe it via a loop group method
similar to \cite{DPW} for two different base points.

For relating the normalized frames obtained by the DPW-method  for two different basepoints
we will use the notation applied in the last section.

We consider some symmetric space $ S = G/K_0$, but, in this section, we  will never change
its realization as the quotient space $G/K_0$.

As in the previous section we consider two  base points, $z_0, z_2 \in \D,$ where $\D$ is, up to the last subsection of this section, the open  unit disk or the whole complex plane.
The case $\D = S^2$ will be discussed in the last subsection of this section.

 As in the last section we consider a harmonic map
$f_0 : \D \rightarrow G/K_0$ satisfying $f(z_0) = eK_0$.

In this setting we can apply the standard DPW-procedure for $f_0$ and $z_0$. as discussed 
in the first four subsections of the last section.

It is sometimes  useful to consider two different  base points. 

To avoid any confusion with 
$z_1$, used in the previous section, we will  denote  the new second base point by $z_2$ in this section.

\subsection{From immersions to extended frames}

Under the general assumptions  just stated above there exists some $ \mathring{g}  \in G$ satisfying 

\begin{equation} \label{defmathring}
f_0(z_2) = \mathring{g}.f_0(z_0) .
\end{equation}

Note that if  two group elements, $\mathring{g}, \check{g}$, both satisfy the relation above, then 
$\check{g} = k_2 \mathring{g} k$ for some $k \in K_0$ and some $k_2  \in G$  fixing $f(z_2).$

For convenience we set $$ f_2(z) = \mathring{g}.f_0(z).$$

Clearly, $f_2$ is a harmonic map which  satisfies $f_2(z_2) = e K_0.$

Fixing some $\mathring{g}$ as above, we 
 carry out the classical DPW-procedure
 for $f_2$ as explained in the last section, but with $z_0$ replaced by $z_2$.
 
 We thus have two pairs, $(f_0,z_0)$ and $(f_2,z_2)$. Since $f_0(z_0) = eK$ as well as 
 ${f_2}(z_2) = eK,$ we can work out the standard DPW-Procedure relative to $S = G/K_0$
 for each of these pairs.
 
 The main issue in this section is to discuss the relation between these two descriptions.

If $F_0$ is a frame for $f_0$ satisfying $F_0(z_0,\overline{z_0}) = I$ and $f _0= F_0.eK = F_0.f(z_0),$ and if 
 ${F_2}$ is a frame for ${f_2}$ satisfying ${F_2}(z_2,\overline{z_2}) = I$ and 
 ${f_2} = {F_2}.eK_0 ={F_2}.f(z_2),$  then (\ref{defmathring}) implies 
 $ {f_2} = {F_2}.eK_0 = \mathring{g}F_0.eK_0,$ from which we derive 
 
 \begin{equation} \label{relateFandhat{F}}
 {F_2} = \mathring{g} F_0 k_0,
 \end{equation}
where $k_0 $ takes values in $K$.

From this  we obtain, as in the last section,  an extended frame ${F_2}(z, \bar z, \lambda)$ satisfying 
${F_2}(z_2, \overline{z_2}, \lambda) = I$  for all $ \lambda \in S^1$ and 
${F_2}(z, \bar z, \lambda = 1) = {F_2}(z, \bar z).$

Once we have defined the family of frames ${F_2}(z, \bar z, \lambda)$ we can define the associated family of harmonic maps from $\D$ to $G/K_0:$

$${f_2}(z, \bar z, \lambda) = {F_2}(z, \bar z, \lambda). eK_0$$

The primary objects, the associated family of ${f_2}$ and the extended frame of ${f_2},$ normalized at $z_2$, 
 can then be obtained by

\begin{equation}
{f_2}(z, \bar z,\lambda)  = \mathring{g}(\lambda) f_0(z, \bar z, \lambda),
\end{equation}

and

\begin{equation}  \label{relateextframes}
 {F_2}(z, \bar z,\lambda)  = \mathring{g}(\lambda) F_0(z, \bar z, \lambda) k_0(z, \bar z).
 \end{equation}

for some $\mathring{g} (\lambda) = \Lambda G_{\sigma}$ satisfying  $\mathring{g} (\lambda = 1) = \mathring{g}.$

Substituting $z_2$ into (\ref{relateextframes}) we derive from the last equation

\begin{equation}
\mathring{g} (\lambda)^{-1} = F_0(z_2, \overline{z_2},\lambda) k_0(z_2, \overline{z_2}).
\end{equation}

and substituting  $z_0$ for $z$   into (\ref{relateextframes}) yields

\begin{equation}
 F_2(z_0, \overline{z_0},\lambda) = \mathring{g}(\lambda)  k_0(z_0, \overline{z_0}).
\end{equation}

\subsection{From frames to normalized potentials}

 So far, the relation between the frames $F_0$ and ${F_2},$  based at $z_0$ and at $z_2$ respectively,
is very simple.

The usefulness of the DPW-method is related to the possibility to describe "potentials" and in particular
to describe normalized potentials.

The combination of the  two statements just above produces a complicated result.
It seems that, in a sense, these two statements are not "simply compatible", opposite to the situation in the last section.

We thus discuss the second statement above in the present situation.
For the frame $F_0$ this is the decomposition

\begin{equation}
F_0(z, \bar z, \lambda) = F_{0,-}(z, \lambda) F_{0,+}(z, \bar z, \lambda),
\end{equation}

and for the frame ${F_2}$  this is the decomposition

\begin{equation}
{F_2}(z, \bar z, \lambda) = {F}_{2,-}(z, \lambda) {F}_{2,+}(z, \bar z, \lambda),
\end{equation}

where the normalized frames  ${F}_{0,-}(z, \lambda)$ and $ {F}_{2,-}(z, \lambda)$ are, 
as functions of $z,$ meromorphic and ,
as functions of $\lambda,$ of the form $I + \mathcal{O}(\lambda^{-1})$, whence the 
decompositions are unique.

The normalized potentials $\eta_{0,-}$ and ${\eta}_{2,-}$  then are the Maurer-Cartan forms of  
$F_{0,-}(z, \lambda)$ and $ F_{2,-}(z, \lambda)$
respectively.

Since  we know $ {F_2}(z, \bar z,\lambda)  = \mathring{g}(\lambda) F_0(z, \bar z, \lambda) k_0(z, \bar z),$
we obtain (by using the unique Iwasawa decmposition)

\begin{equation}
{F_{2,-}}(z, \bar z,\lambda) = ( \mathring{g}(\lambda) F_0(z, \bar z, \lambda) )_-.
\end{equation}

Moreover, rewriting  
$$ {F_2}(z, \bar z,\lambda)  = \mathring{g}(\lambda) F_0(z, \bar z, \lambda)  k_0(z, \bar z),$$
 as

\begin{equation} \label{longequ}
 {F}_{2,-}(z, \bar z,\lambda)   = \mathring{g}(\lambda) F_{0,-}(z, \bar z, \lambda) \cdot (F_{0,+}(z, \bar z, \lambda) k_0(z, \bar z) {F}_{2,+}(z, \bar z,\lambda)^{-1})
 \end{equation}
 we infer

\begin{proposition}
If   $\mathring{g}(\lambda)$ has the Birkhoff decomposition 
$ \mathring{g}(\lambda)  = \mathring{g}_-(\lambda)   \mathring{g}_+(\lambda),$ 

then equation  (\ref{longequ})
can be rephrased involving the dressing  of  $F_{0,-}(z, \bar z, \lambda)$  by  $\mathring{g}_+ (\lambda) $

 $$ {F_{2,-}}(z, \bar z,\lambda)   = \mathring{g}_- (\lambda)  \cdot (\mathring{g}_+(\lambda) F_{0,-}(z, \bar z, \lambda) \mathring{g}_+(\lambda)^{-1}) \cdot Q_+(z, \bar z, \lambda)
 = \mathring{g}_-(\lambda)( F_{0,-}(z, \bar z, \lambda) \sharp \mathring{g}_+(\lambda)^{-1})_-.$$
 \vspace{2mm}
 
 Here ``$\sharp$" denotes dressing and we abbreviate 
 
 $$Q_+(z, \bar z, \lambda) = \mathring{g}_+(\lambda) \cdot  (F_{0,+}(z, \bar z, \lambda) k_0(z, \bar z) {F}_{2,+}(z, \bar z,\lambda)^{-1}).$$
 \end{proposition}
 \begin{proof}
 The first equality just gives an interpretation of (\ref{longequ}). From this   the claim follows.
 \end{proof}

\begin{remark}
The last proposition can also be stated for $\mathring{g}$ which have a general Birkhoff decomposition. 
But the result looks more complicated than the one above. Independent of this, the case discussed above 
is the generically occurring one.

\end{remark}

\begin{corollary}
The normalized potential of ${f_2}$ is equal to the Maurer-Cartan form of 

$\mathring{g}_-(\lambda)( F_{0,-}(z, \bar z, \lambda) \sharp \mathring{g}_+(\lambda)^{-1})_-,$
where $\mathring{g}(\lambda)  = \mathring{g}_+(\lambda)  \mathring{g}_-(\lambda) \in G$.
\end{corollary}

 \subsection{ From normalized potentials to the harmonic maps relative to  different base points.}
 
 It is fairly easy to reverse the argument of the last section.
 
 \begin{proposition}
 Let $f_0, {f_2}: \D \rightarrow G/K_0 $ be harmonic maps which for $z_0,z_2 \in \D$ satisfy  
 $f_0(z_0) = {f_2}(z_2) = eK_0$ and 
 whose extended frames we denote by $F_0$ and ${F_2}$.
 We assume $F_0(z_0, \overline{z_0}, \lambda) = I$ and ${F_2}(z_2, \overline{z_2}, \lambda) = I$
 We assume moreover that
 
 $\bullet$ there exists some $\mathring{g}(\lambda)  \in  \Lambda G^\C_{\sigma}$ 
 for which we have the  Birkhoff decomposition 
 
  $\mathring{g}(\lambda)  = \mathring{g}_-(\lambda)  \mathring{g}_+(\lambda) \in G$,
  
  $\bullet$ the normalized frame of ${f_2}$ is equal to
  $\mathring{g}_-(\lambda)( F_{0,-}(z, \bar z, \lambda)  \sharp \mathring{g}_+(\lambda)^{-1})_-,$
  where $\sharp$ denotes "dressing" and the last term is of the form  $I + \mathcal{O}(\lambda^{-1})$.
  Then for the extended frames we obtain the relation
  $$  {F_2}(z, \bar z,\lambda)  = \mathring{g}(\lambda) F_0(z, \bar z, \lambda) k_0(z, \bar z).$$
  
  And for the associated families we obtain
  
  $${f_2}(z, \bar z, \lambda) =  \mathring{g}(\lambda) f_0(z, \bar z, \lambda).$$
  
 Furthermore, $ \mathring{g}(\lambda =1) f(z_0)  = f(z_2).$
  
 \end{proposition}
 
 \begin{proof}
 By assumption we have ${F_2}_- (z, \bar z, \lambda) = \mathring{g}_-(\lambda)( F_0(z, \bar z, \lambda)_-  \sharp 
 \mathring{g}_+(\lambda)^{-1})_- ,$ whence we obtain
\begin{small}
\begin{equation*} 
\begin{split}
&{F_2} (z, \bar z, \lambda)=\\
& \mathring{g}_-(\lambda)(  \mathring{g}_+(\lambda)F_{0,-}(z, \bar z, \lambda) 
 \mathring{g}_+(\lambda)^{-1})_- \cdot \\
 &(  \mathring{g}_+(\lambda)F_{0,-}(z, \bar z, \lambda)  
 +\mathring{g}_+(\lambda)^{-1})_+
 \cdot 
 ((  \mathring{g}_+(\lambda)F_0(z, \bar z, \lambda)_-  \mathring{g}_+(\lambda)^{-1})_+^{-1} {F}_+^{-1}(z, \bar z, \lambda)
 = \\
 & \mathring{g}_-(\lambda)(  \mathring{g}_+(\lambda)F_{0,-}(z, \bar z, \lambda) V_+ = 
 \mathring{g}F_0 W_+.
 \end{split}
 \end{equation*}
 \end{small}
 
 Since $F_0, {F_2}$ and $\mathring{g}$ are contained in $\Lambda G_\sigma,$ we conclude $W+ \in K_0$.
 This proves the first claim. 
 
 The second claim follows by application to the point $eK_0 \in G/K_0.$
 
 The last claim is a consequence of the second one obtained by evaluation at $z_2$.
  \end{proof}
 \subsection{The relation between the harmonic maps into the compact duals}
 
 In this section we will use the notation of the previous sections freely.
 
 We consider, in particular, again a harmonic map $f_0:\D \rightarrow G/K_0$ and two base points $z_0$ and $z_2$.
 
 We choose again $\mathring{g} \in G$ such that  $\mathring{g}.f_0(z_0) = f_2(z_2). $ holds and define 
 ${f_2} = \mathring{g} f_0.$ Then the extended frames ${F_2}$ and $F_0$ satisfy
 
 \begin{equation} \label{relationframes}
  {F_2}(z, \bar z,\lambda)  = \mathring{g}(\lambda) F_0(z, \bar z, \lambda) k_0(z, \bar z).
  \end{equation}
   
   Recall that $F_0$ attains the value $I$ at $z = z_0 \in \D$, while ${F_2}$ attains the value $I$ at $z = z_2$.
   
   Next we want to determine the dual harmonic maps into the compact dual of $G/K_0$.
   Note, in general there may be different compact duals, but we have a specific procedure to determine them, whence we only have one compact dual for both harmonic maps, namely the one described in Theorem 1.1
   \cite{DoWa:Dual}
   
   In the spirit of the procedure outlined in the previous sections above we consider the global 
   Iwasawa splittings relative to the maximal compact subgroup $U$ defined in  
    section 3.1 of \cite{DoWa:Dual} in the construction of the compact dual symmetric space.
   
   \begin{enumerate}
   \item $F_0= F_{0,U} V_{0,+},$\\
   \item ${F_2} = {F}_{2,U} {V}_{2,+}.$
   \end{enumerate}
   
 In view of (\ref{relationframes}) this yields

 \begin{equation} \label{relatedualframes}
 F_{2,U} = \mathring{g} F_{0,U} W_+,
 \end{equation}
  with some $W_+ \in \Lambda^+G^\C_\sigma$.
  
  This can be rephrased as follows:
  
  \begin{proposition} Let $f_2(z_2) = \mathring{g}.f_0(z_0) , $ with  $ \mathring{g} \in G$ as in  
  (\ref{defmathring}) above.
  Then the dual frame $ {F}_{2,U} $ of ${F_2}$ can be obtained from  the dual frame
  $ F_{0,U} $ of $F_0$ by dressing by  $ \mathring{g}_+.$
  \end{proposition}
  
  \begin{corollary}
  From the two Iwasawa splittings stated just above we derive that $F_0$ and $F_{0,U}$ as well as 
  ${F_2}$  and  ${F}_{2,U}$ are uniquely determined from each other. Therefore, $F_0 =  {F_2}$
  if and only if $F _{0,U}=  {F}_{2,U}.$ 
  \end{corollary}
  
  \begin{remark}
  Now  (\ref{relationframes}) implies that this holds if and only
   if $\mathring{g}$ is in the isotropy group of the dressing action at the point $F_0$ resp. $F_{0,U}$.
  It is known that in many cases this isotropy group is fairly small, sometimes only $\{\pm I\}$.
   In general, however, the dressing isotropy group is fairly big. As a consequence, 
   one should expect that $F_0$ and ${F_2}$ are generally different.
  An example for this is presented in section 6 of \cite{DoWa:Dual}. 
  \end{remark}
 
\subsection{The case $\D = S^2$}

We follow the loop group approach for harmonic maps defined on $S^2$ and consider two base points 
$z_0 \neq z_2$ and also choose two points $p_0 \neq p_2$  in $S^2$ satisfying 
$\{z_0, z_2 \} \cap \{z_0, z_2 \} = \emptyset.$ Now we proceed as in Section \ref{sectionS^2}.
We leave the details to the reader.


\section{Appendix}

For the convenience of the reader we list here the basic notation and the basic results about the loop group method.
More information about  the loop group method 
can be found in \cite{DPW}, \cite{DoHa:sym1} , \cite{DoHa:sym2} \cite{DoWa:sym1}.

Denote the interior of the unit disk by $\mathbf{D}=\{\lambda\in \mathbb{C}| |\lambda|<1\}$ and by $\E$  the exterior of the unit disk, 
$\E=\{\lambda\in \mathbb{C}| |\lambda|>1\} \cup {\infty}.$

Since we are primarily interested in groups and loop groups related to $G^\C$, we  write down the conditions below for simply-connected complex (matrix) Lie groups only.

Set 
\begin{align*}
\Lambda G_{\sigma}^{\mathbb C}&=\{ g: S^1\rightarrow G^{\mathbb C} | g
\text{  has finite Wiener norm},\, g(\epsilon \lambda)=\sigma (g(\lambda))\}, \\
\Lambda^{+} G^{\mathbb C}_{\sigma}&=\{g\in \Lambda G^{\mathbb{C}}_{\sigma} | \, g \text{ extends holomorphically to } \mathbf{D}, g(0)\in K^{\mathbb C}\},\\
\Lambda^{+}_B G_{\sigma}^{\mathbb C}&=\{ g\in \Lambda^{+} G_{\sigma}^{\mathbb C} |\, 
g(0)\in B \},\\
\Lambda^{-} G^{\mathbb C}_{\sigma}&=\{ g\in \Lambda G^{\mathbb{C}}_{\sigma} | \, g \text{ extends holomorphically to } \mathbb{E}, g(\infty)\in K^{\mathbb C}\},\\
\Lambda^{-}_{*} G^{\mathbb C}_{\sigma}&=\{ g\in \Lambda^{-} G^{\mathbb C}_{\sigma} | \, g(\infty)=e\},
\end{align*}
where $K^{\mathbb C} = KB$ is a fixed Iwasawa decomposition of $K^{\mathbb C}$.

We will always equip $\Lambda G_{\sigma}^{\mathbb C}$ with 
the Wiener topology of absolute convergence of the Fourier coefficients. Then the group
$\Lambda G^{\mathbb C}_{\sigma}$ becomes a complex Banach Lie group  with Lie algebra
\begin{equation*}
\Lambda \mathfrak{g}^{\mathbb{C}}_{\sigma}:=\{\xi: S^1 \rightarrow \mathfrak{g}^{\mathbb C}|
 \sigma(\xi(\lambda))=\xi(\epsilon \lambda)\}.
\end{equation*}

If $\xi \in \Lambda \mathfrak{g}^{\mathbb C}_{\sigma}$, its Fourier decomposition is
$$\xi=\sum_{l \in \mathbb{Z}} \lambda^l\xi_l, \quad \xi_l \in \mathfrak{g}_l$$
and the Lie subalgebras of $\Lambda \mathfrak{g}^{\mathbb C}_{\sigma}$ corresponding to the subgroups
$\Lambda G_{\sigma}$, $\Lambda^{+} G^{\mathbb C}_{\sigma}$ and $\Lambda^{-} G^{\mathbb C}_{\sigma}$ are
\begin{align*}
\Lambda \mathfrak{g}_{\sigma}&=\Lambda \mathfrak{g}^{\mathbb C}_{\sigma}\cap \mathfrak{g},
\\
\Lambda^{+}\mathfrak{g}^{\mathbb C}_{\sigma}&=\{\xi\in \Lambda \mathfrak{g}^{\mathbb C}_{\sigma} |
\xi_l=0 \text{ for } l<0, \xi_0\in \mathfrak{k}\},\\
\Lambda^{-}\mathfrak{g}^{\mathbb C}_{\sigma}&=\{\xi\in \Lambda \mathfrak{g}^{\mathbb C}_{\sigma} |
\xi_l=0 \text{ for } l>0, \xi_0\in \mathfrak{k}\}.
\end{align*}
Similar conditions hold for the remaining two Lie algebras.

We finish this section by quoting the two splitting theorems which are of 
crucial importance for the application of the loop group method.
The first of these theorems is due to Birkhoff, who invented it for the loop group of $GL(n,\C)$ in an attempt to solve Hilbert's 21'{st} problem. 

\begin{theorem}[Birkhoff Decomposition]
Let $G$ be a simply connected  real Lie group. Then the multiplication $\Lambda^{-}_{*} G^{\mathbb C}_{\sigma} \times \Lambda^{+} G^{\mathbb C}_{\sigma} \rightarrow \Lambda G^{\mathbb C}_{\sigma}$ is a 
complex analytic diffeomorphism onto the open, connected  and dense subset 
$\Lambda^{-}_{*} G^{\mathbb C}_{\sigma}\cdot \Lambda^{+} G^{\mathbb C}_{\sigma}$ of  $\Lambda G^{\mathbb C}_{\sigma},$
called the big (left)  Birkhoff cell. 

In particular, if $g\in \Lambda G^{\mathbb C}_{\sigma}$ is contained in the big cell,  then
$g$ has a unique decomposition $g=g_{-}g_{+}$, where $g_{-}\in \Lambda_{*}^{-} G_{\sigma}^{\mathbb C}$ 
and $g_{+}\in \Lambda^{+}G^{\mathbb C}_{\sigma}$.
\end{theorem}

The second crucial loop group splitting theorem is the following

\begin{theorem}[Iwasawa decomposition]

Let $G$ be a simply connected  real Lie group. Then the multiplication map 
$$\Lambda G_{\sigma}\times \Lambda^{+}_B G_{\sigma}^{\mathbb C} 
\rightarrow \Lambda G_{\sigma}^{\mathbb C}$$
is a real-analytic diffeomeorphism onto the open, connected and dense subset
$\Lambda G_{\sigma} \cdot \Lambda^{+}_B G_{\sigma}^{\mathbb C} $ of  $.\Lambda G_{\sigma}^{\mathbb C}$
\end{theorem}

This result is well known for untwisted loop groups (see, e.g. Pressley-Segal \cite{PS})  and was extended to the twisted setting in  \cite{DPW}.  

More information concerning the loop group method can be found in \cite{DPW} and papers referring to
\cite{DPW}.

{\footnotesize
\def\refname{Reference}

}

\vspace{2mm}

{\footnotesize }

Josef F. Dorfmeister

Fakult\" at f\" ur Mathematik,

TU-M\" unchen, Boltzmannstr. 3,

D-85747, Garching, Germany

{\em E-mail address}: josef.dorfmeister@tum.de\\

\end{document}